\documentclass[reqno,A4paper]{amsart}

\usepackage{amsmath}
\usepackage{float}
\usepackage{amssymb}
\usepackage{amsthm}
\usepackage{enumerate}
\usepackage{mathrsfs} %%\mathcal classic,  \mathscr decorative
\usepackage{eqlist}
\usepackage{array}
\usepackage{epsfig}
\usepackage{amsmath}
\usepackage{amsthm}
\usepackage{amssymb}
\usepackage{epstopdf}
\usepackage{amstext,color}
\usepackage{array,multirow}
\usepackage{graphicx}
\usepackage{enumitem}
\usepackage[vlined,ruled,linesnumbered]{algorithm2e}
\theoremstyle{plain}
\newtheorem{theorem}{Theorem}[section]
\newtheorem{lemma}{Lemma}[section]

\theoremstyle{definition}

\newtheorem{example}{\textit{Example}} %\textit for ``Example'' is required

\numberwithin{equation}{section}
\begin{document}
	\title[A new sequence of operators involving Laguerre polynomials]%
	{Approximation by a new sequence of operators involving Laguerre polynomials}
	\author[K. Kumar \and N. Deo \and D. K. Verma]%
	{Kapil Kumar* \and Naokant Deo \and Durvesh Kumar Verma}
	
	\newcommand{\acr}{\newline\indent}
	
	\address{\llap{*\,} Department of Mathematics\acr
		Delhi Technological University\acr
		Delhi-110042\acr
		INDIA}
	\email{kapil\_2k19phdam502@dtu.ac.in, kapil@hrc.du.ac.in}
	\address{\llap{\,} Department of Mathematics\acr
		Delhi Technological University\acr
		Delhi-110042\acr
		INDIA}
	\email{naokantdeo@dce.ac.in}
	
	\address{\llap{\,}Department of Mathematics\acr
		Miranda house \acr
		University of Delhi\acr
		Delhi-110007\acr
		INDIA}
	\email{durvesh.kv.du@gmail.com}
	
	%%\acr is not required (if you do not need to see a column);
	%%in our style \\ makes a column automatically
	
	%\thanks{This work was supported by ...... Grant No. ...., etc. (optional)}
	
	 %Secondary is optional
	
	\begin{abstract}
	%	This paper presents a new integral type of construction of operators by involving the P\u{a}lt\u{a}nea basis function and orthogonal modified Laguerre polynomials. It can be used to approximate the functions over the interval $[0,\infty)$. Further, the moments are established for the proposed operators, and the approximation properties of the operators are obtained using the universal Korovkin’s theorem. The convergence properties for proposed operators are investigated by using the modulus of continuity, the second-order modulus of smoothness, Peetre’s K-function, and the Lipschitz class. A Voronovskaja-zype asymptotic formula is also admitted. The convergence of the operators is also discussed in weighted spaces of functions, and the approximation is estimated with the help of a weighted modulus of continuity. Ultimately, to validate the findings, we employ numerical illustrations and visual depictions.
 This paper offers a newly created integral approach for operators employing the orthogonal modified Laguerre polynomials and P\u{a}lt\u{a}nea basis. These operators approximate the functions over the interval $[0,\infty)$. Further, the moments are established for the proposed operators, and the universal Korovkin's theorem is used to derive the approximation properties of the operators. We examine convergence using a variety of analytical methods, including the Lipschitz class, Peetre's K-functional, the second-order modulus of smoothness, and the modulus of continuity. Moreover, an asymptotic formula associated with the Voronovskaja-type is established. The approximation is estimated through the weighted modulus of continuity, and convergence of the proposed operators in weighted spaces of functions is investigated as well. Ultimately, we employ numerical examples and visual representations to validate the theoretical findings.

	\end{abstract}
	
	\maketitle
	\section{Introduction}
For the weight function $\Phi^{(\alpha)}(x)=x^{\alpha}\exp{(-x)}, \alpha>-1$, the Laguerre polynomials $\mathcal{L}_{\kappa}^{(\alpha)}$ are orthogonal  over the interval $[0,\infty)$. These polynomials can be expressed by using the Rodrigues formula (cf.\cite{rainville1960special} p. 203-204]):\\
$$\mathcal{L}_{\kappa}^{(\alpha)}(x)=\frac{\Phi^{(-\alpha)}(-x)}{\kappa!}D^{k}(\Phi^{(k+\alpha)}(-x))=\sum_{\iota=0}^{\kappa}\frac{(-1)^{\iota}}{\iota!}\binom{\kappa+\alpha}{\kappa-\iota}x^{\iota}, \mbox{where}\;D=\frac{d}{dx}.$$  
$\mathcal{L}_{\kappa}^{(\alpha)}$ can be generated with the help of the following generating function:
\begin{equation*}
	(1-z)^{-\alpha-1}\exp{\left(-\frac{xz}{1-z}\right)}=\sum_{\kappa=0}^{\infty}\mathcal{L}_{\kappa}^{(\alpha)}(x)z^{\kappa}
\end{equation*}
Recently, Gupta (\cite{gupta2024new}) introduced the operators based on orthogonal modified Laguerre polynomials over the interval $[0,\infty)$. These operators are as follows:
\begin{equation*}
	P_{\eta,\kappa}^{\alpha}(\phi;x)=\exp{\left(\frac{-\eta x}{2}\right)}2^{-\alpha-1}\sum_{\kappa=0}^{\infty}2^{-\kappa}\mathcal{L}_{\kappa}^{(\alpha)}\left(\frac{-\eta x}{2}\right)\Phi\left(\frac{\kappa}{\eta}\right),\;\;\alpha>-1.
\end{equation*}
Let $\Phi:[0,\infty)\rightarrow\mathbb{R}$ be integrable function. For $\rho>0$ and $\beta>0$, P\u{a}lt\u{a}nea \cite{kapilkumar} defined the modification of Sz\'{a}sz-Mirakjan operators as\\
$$\mathcal{B}_{\rho,\kappa}^{(\beta)}(\Phi;x)=\exp{(-\rho x)}\left(\Phi(0)+\sum_{\kappa=1}^{\infty}\frac{(\rho x)^{\kappa}}{\kappa!}\int_{0}^{\infty}\frac{\rho\beta\exp{(-\rho\beta z)}(\rho \beta z)^{\kappa\beta-1}}{\Gamma(\kappa\beta)}\Phi(z)dz\right).$$\\
Cheney and Sharma \cite{cheney1964bernstein} proposed operators based on Laguerre polynomials that are intended to explore the variation-diminishing properties and convergence. {\"O}ks{\"u}zer et al. \cite{oksuzer2018rate} introduced B{\'e}zier variant of operators involving Laguerre polynomials and studied rate of convergence for these operators.  {\"O}zarslan \cite{ozarslan2007q} proposed the q -Laguerre type positive linear operators and examined the approximation properties of these operators. Deshwal et al. \cite{deshwal2018pointwise} constructed B{\'e}zier variant of an operator based on Laguerre polynomials. The author examined the degree of approximation and rate of convergence for these operators.
 Much effort has been done recently to integral type modification of positive linear operators applying P\u{a}lt\u{a}nea basis. Verma and Gupta \cite{verma2015approximation} proposed the Jakimovski–Leviatan–P\u{a}lt\u{a}nea operators. They estimated approximation properties for the operators. Ansari et. al \cite{ansari2019approximation} constructed  modified P\u{a}lt\u{a}nea operators with Gould–Hopper
polynomials and estimated the approximation properties by using the weighted modulus of continuity. Mishra and Deo \cite{mishra2024approximation} introduced Apostol-Genocchi-P\u{a}lt\u{a}nea operators and studied the asymptotic type results as the Voronovskaya theorem and quantitative Voronovskaya theorem.
To see more work relevant to this area, one may refer (\cite{agrawal2021approximation,mursaleen2019approximation,neer2017quantitative,kapmursaleen2019approximation,aral2016quantitative,kumar2022approximation,agratini2021approximation,arunkajla,biricik2024laguerre,moak1981q}).\\
Motivated by the work above, we introduce a new class of operators adopting modified Laguerre polynomials and P\u{a}lt\u{a}nea basis function over the interval $[0,\infty)$  as follows:\\
%If $\Phi\in \mathcal{C}[0,\infty)=\{\Phi:\Phi\;\mbox{is continuous and bounded over}\;[0,\infty)\}$,
If $\Phi:[0,\infty)\rightarrow\mathbb{R}$ is integrable function, then
\begin{equation}\label{eqn1}
	\mathcal{R}_{\eta,\kappa}^{(\alpha,\beta)}(\Phi;x)=\exp{\left(\frac{-\eta x}{2}\right)}2^{-\alpha-1}\sum_{\kappa=0}^{\infty}2^{-\kappa}\mathcal{L}_{\kappa}^{(\alpha)}\left(\frac{-\eta x}{2}\right)\int_{0}^{\infty}\mathcal{I}_{\kappa,\eta}^{\beta}(z)\Phi(z)dz, 
\end{equation}
where $\mathcal{I}_{\kappa,\eta}^{\beta}(z)=\frac{\eta\beta\exp{(-\eta\beta z)}(\eta \beta z)^{\kappa\beta-1}}{\Gamma(\kappa\beta)}$ and $\beta>0$.\\
The main goal of this study is to create a new class of operators. This is an overview of the profile of the paper being presented, we constructed a newly integral type operator over the interval $[0,\infty)$. The moments associated with the proposed operators are addressed in section 2. In section 3, we used Peetre's K-functional, the known Voronovskaja type asymptotic formula, the modulus of continuity, the Lipschitz class, and the second-order modulus of smoothness to examine the convergence traits of the proposed operators. Section 4 discussed the convergence of the operators in weighted spaces and estimated the approximation by using the weighted modulus of continuity. We use both visual and numerical illustrations in the last section to support the results. 
\begin{lemma}\label{lemma1}
	\cite{kapilkumar}. For $\varrho_{\jmath}(z)=z^{\jmath}$ and $\jmath=0,1,2,\hdots$, we know
	\begin{align*}
		\int_{0}^{\infty}\mathcal{I}_{\kappa,\eta}^{\beta}(z)\varrho_{\jmath}(z)dz&=\frac{(\kappa\beta)_{\jmath}}{(\eta\beta)^{\jmath}},
	\end{align*}
	where $(x)_{\jmath}=\prod_{s=1}^{\jmath}(x+s-1), (x)_{0}=1.$
\end{lemma}
\begin{lemma}\label{lemma2}
	\cite{gupta2024new}. For the operators $P_{\eta,\kappa}^{\alpha}(\phi;x)$ and $\jmath=0,1,2,3,4$ , we get
	\begin{align*}
		P_{\eta,\kappa}^{\alpha}(\varrho_{0}(z);x)&=1,\\
		P_{\eta,\kappa}^{\alpha}(\varrho_{1}(z);x)&=x+\frac{1+\alpha}{\eta},\\
		P_{\eta,\kappa}^{\alpha}(\varrho_{2}(z);x)&=x^{2}+\frac{x(2\alpha\eta+5\eta)+\alpha^{2}+4\alpha+3}{\eta^{2}},\\
		P_{\eta,\kappa}^{\alpha}(\varrho_{3}(z);x)&=x^{3}+\frac{x^{2}(3\alpha\eta^{2}+12\eta^{2})+x(3\alpha^{2}\eta+21\alpha\eta+31\eta)+\alpha^{3}+9\alpha^{2}+21\alpha+13}{\eta^{3}},\\
		P_{\eta,\kappa}^{\alpha}(\varrho_{4}(z);x)&=x^{4}+\frac{x^{3}(4\alpha\eta^{3}+22\eta^{3})+x^{2}(6\alpha^{2}\eta^{2}+60\alpha\eta^{2}+133\eta^{2})}{\eta^{4}}\\
		&\quad+\frac{x(4\alpha^{3}\eta+54\alpha^{2}\eta+208\alpha\eta+233\eta)+\alpha^{4}+16\alpha^{3}+78\alpha^{2}+138\alpha+75}{\eta^{4}}.
	\end{align*}
	\begin{lemma}\label{lemma3}
		For $\varrho_{\jmath}(z)=z^{\jmath}$ and $\jmath=0,1,2,3,4$, then
		\begin{align*}
			\mathcal{R}_{\eta,\kappa}^{(\alpha,\beta)}(\varrho_{0}(z);x)&=1,\\
			\mathcal{R}_{\eta,\kappa}^{(\alpha,\beta)}(\varrho_{1}(z);x)&=x+\frac{1+\alpha}{\eta},\\
			\mathcal{R}_{\eta,\kappa}^{(\alpha,\beta)}(\varrho_{2}(z);x)&=x^{2}+\frac{x(2\alpha\beta\eta+5\beta\eta+\eta)+\beta(\alpha^{2}+4\alpha+3)+\alpha+1}{\eta^{2}\beta},\\
			\mathcal{R}_{\eta,\kappa}^{(\alpha,\beta)}(\varrho_{3}(z);x)&=x^{3}+\frac{x^{2}(3\alpha\beta^{2}\eta^{2}+12\beta^{2}\eta^{2}+3\beta\eta^{2})+x(3\alpha^{2}\beta^{2}\eta+21\alpha\beta^{2}\eta+31\beta^{2}\eta+6\alpha\beta\eta+15\beta\eta+2\eta)}{\eta^{3}\beta^{2}},\\
			&\quad+\frac{\alpha^{3}\beta^{2}+9\alpha^{2}\beta^{2}+21\alpha\beta^{2}+13\beta^{2}+3\alpha^{2}\beta+12\alpha\beta+9\beta+2\alpha+2}{\eta^{3}\beta^{2}},\\
			\mathcal{R}_{\eta,\kappa}^{(\alpha,\beta)}(\varrho_{4}(z);x)&=x^{4}+\frac{x^{3}(4\alpha\beta^{3}\eta^{3}+6\beta^{2}\eta^{3}+22\beta^{3}\eta^{3})}{\eta^{4}\beta^{3}}\\
			&\quad+\frac{x^{2}(11\beta\eta^{2}+72\beta^{2}\eta^{2}+18\alpha\beta^{2}\eta^{2}+133\beta^{3}\eta^{2}+60\alpha\beta^{3}\eta^{2}+6\alpha^{2}\beta^{3}\eta^{2})}{\eta^{4}\beta^{3}}\\
			&\quad+\frac{x(4\alpha^{3}\beta^{3}\eta+54\alpha^{2}\beta^{3}\eta+208\alpha\beta^{3}\eta+233\beta^{3}\eta+18\alpha^{2}\beta^{2}\eta+126\alpha\beta^{2}\eta)}{\eta^{4}\beta^{3}}\\
			&\quad+\frac{x(186\beta^{2}\eta+22\alpha\beta\eta+55\beta\eta+6\eta)}{\eta^{4}\beta^{3}}+\frac{\alpha^{4}\beta^{3}+16\alpha^{3}\beta^{3}+78\alpha^{2}\beta^{3}+138\alpha\beta^{3}}{\eta^{4}\beta^{3}}\\
			&\quad+\frac{75\beta^{3}+6\alpha^{3}\beta^{2}+54\alpha^{2}\beta^{2}+126\alpha\beta^{2}+78\beta^{2}+11\alpha^{2}\beta+44\alpha\beta+33\beta+6\alpha+6}{\eta^{4}\beta^{3}}.
		\end{align*}
	\end{lemma}
\end{lemma}
\begin{proof}
	By Lemma (\ref{lemma1}) and  Lemma (\ref{lemma2}), we can easily calculate following moments for the proposed operators.
\end{proof}
\begin{lemma}\label{lemma4}
	Let $\mu_{\eta,\jmath}(x):=\mathcal{R}_{\eta,\kappa}^{(\alpha,\beta)}((\varrho_{1}(z)-x\varrho_{0}(z))^{\jmath};x)$ and $\jmath=1,2,4$, we calculate easily by Lemma (\ref{lemma3}), we get
	\begin{align*}
		\mu_{\eta,1}(x)&= \frac{1+\alpha}{\eta},\\
		\mu_{\eta,2}(x)&=\frac{x(3\beta+1)}{\beta\eta}+\frac{\alpha^{2}\beta+4\alpha\beta+3\beta+\alpha+1}{\beta\eta^{2}},\\
		\mu_{\eta,4}(x)&=\frac{3x^{2}(9\beta^{2}+6\beta+1)}{\beta^{2}\eta^{2}}+\frac{x(124\alpha\beta{3}+181\beta^{3}+6\alpha^{2}\beta^{2}+78\alpha\beta^{2}+150\beta^{2}+14\alpha\beta+47\beta+6)}{\beta^{3}\eta^{3}}\\
		&\quad+ \frac{\alpha^{4}\beta{3}+16\alpha^{3}\beta^{3}+78\alpha^{2}\beta^{3}+138\alpha\beta^{3}+75\beta^{3}+6\alpha^{3}\beta^{2}+54\alpha^{2}\beta^{2}+126\alpha\beta^{2}+78\beta^{2}}{\beta^{3}\eta^{4}}\\
		&\quad+\frac{11\alpha^{2}\beta+44\alpha\beta+33\beta+6\alpha+6}{\beta^{3}\eta^{4}}.
	\end{align*}
\end{lemma}
With a normed space $\mathcal{C}_{[0,\infty)}^{\mathcal{B}}=\{\Phi\in \mathcal{C}[0,\infty): \Phi\;\mbox{is bounded over}\;[0,\infty)\}$  that has the norm
\begin{equation}\label{eqn2}
	\|\Phi\|=\sup_{x\in[0,\infty)}|\Phi(x)|
\end{equation}
\begin{lemma}\label{lemma5}
	For each $\Phi\in \mathcal{C}_{[0,\infty)}^{\mathcal{B}}$ and $x\in[0,\infty)$, then
	\begin{equation}\label{eqn3}
		| \mathcal{R}_{\eta,\kappa}^{(\alpha,\beta)}(\Phi;x)|\leq\|\Phi\|.
	\end{equation}
\end{lemma}
\section{Direct Result and Asymptotic Formula}
\begin{theorem}
	Let $\Phi\in\mathcal{C}_{[0,\infty)}^{\mathcal{B}}$. Then, $\lim\limits_{\eta\rightarrow\infty}\mathcal{R}_{\eta,\kappa}^{(\alpha,\beta)}(\Phi;x)=\Phi(x)$ holds uniformly on compact subsets of $[0,\infty)$.
	\begin{proof}
		By Lemma \ref{lemma3}, we can easily see that $\lim\limits_{\eta\rightarrow\infty}\mathcal{R}_{\eta,\kappa}^{(\alpha,\beta)}(\varrho_{\jmath}(z)=z^{\jmath};x)=x^{\jmath}$  for $\jmath=0,1,2,3,4$ and hence the well known Bohman-Korovkin's theorem due to \cite{korovkin1953convergence}, operators $\mathcal{R}_{\eta,\kappa}^{(\alpha,\beta)}$ converge uniformaly on every compact subset of $[0,\infty)$ to $\Phi(x)$.
	\end{proof}
\end{theorem}
In this section, we discuss now explore some important results.\\
For $\Phi\in \mathcal{C}_{[0,\infty)}^{\mathcal{B}}$ , the Peetre's K-functional is given by
\begin{equation*}
	\displaystyle K_{2}(\Phi,\delta)=\inf_{\psi\in W_{C^{2}([0,\infty))}}\{\|\Phi-\psi\|+\delta\|\psi^{''}\|\},\; \mbox{where}\;\delta>0,\;\mbox{and}
\end{equation*}
\begin{equation*}
	W_{C^{2}([0,\infty))}=\{\psi\in \mathcal{C}_{[0,\infty)}^{\mathcal{B}}:\psi^{'},\psi^{''}\in \mathcal{C}_{[0,\infty)}^{\mathcal{B}}\}.
\end{equation*}
%	By DeVore and Lorentz (\cite{devore1993constructive} p.177, Theorem 2.4), there exists a constant $\mathfrak{M}>0$ such that
\begin{equation}\label{eqn4}
	K_{2}(\Phi,\delta)\leq\mathfrak{M}\omega_{2}(\Phi,\sqrt{\delta}),
\end{equation}
\begin{equation*}
	\mbox{where}\; \omega_{2}(\Phi,\sqrt{\delta})=\sup_{0<\epsilon\leq\sqrt{\delta}}\left(\sup_{x,x+\epsilon,x+2\epsilon\in [0,\infty)}|\Phi(x+2\epsilon)-2\Phi(x+\epsilon)+\Phi(x)|\right)
\end{equation*}
is called the second order of modulus of continuity of $\Phi$.
The expression
\begin{equation}\label{eqn5}
	\omega(\Phi,\delta)=\displaystyle\sup_{|z-x|\leq\delta}|\Phi(z)-\Phi(x)|
\end{equation}
is modulus of continuity of $\Phi$, where $x,z\in[0,\infty)$.
\begin{theorem}
	For $\Phi\in \mathcal{C}_{[0,\infty)}^{\mathcal{B}}$ and $x\in[0,\infty)$, then
	\begin{equation*}
		|\mathcal{R}_{\eta,\kappa}^{(\alpha,\beta)}(\Phi;x)-\Phi(x)|
		\leq \mathfrak{M}\omega_{2}(\Phi,\sqrt{\delta/2})+\omega\left(\Phi,\frac{1+\alpha}{\eta}\right),
	\end{equation*} where $\mathfrak{M}>0$ is constant and $\delta=\mu_{\eta,2}(x)+\left(\frac{1+\alpha}{\eta}\right)^{2}$.
\end{theorem}
\begin{proof}
	Firstly, we introduce
	\begin{equation}\label{eqn6}
		\mathcal{H}_{\eta,\kappa}^{(\alpha,\beta)}(\Phi;x)=\mathcal{R}_{\eta,\kappa}^{(\alpha,\beta)}(\Phi;x)-\Phi\left(x+\frac{1+\alpha}{\eta}\right)+\Phi(x).
	\end{equation}
	Let $\psi\in W_{C^{2}([0,\infty))}$, then by Taylor's theorem
	\begin{equation}\label{eqn7}
		\psi(z)=\psi(x)+(z-x)\psi^{'}(x)+\frac{1}{2}\int\limits_{x}^{z}(z-u)\psi^{''}(u)du.
	\end{equation}
	Applying $\mathcal{H}_{\eta,\kappa}^{(\alpha,\beta)}$ on (\ref{eqn7}),
	\begin{align*}
		\mathcal{H}_{\eta,\kappa}^{(\alpha,\beta)}(\psi;x)&=\psi(x)+\psi^{'}(x)\mathcal{H}_{\eta,\kappa}^{(\alpha,\beta)}(z-x;x)+\frac{1}{2}\mathcal{H}_{\eta,\kappa}^{(\alpha,\beta)}\left(\int\limits_{x}^{z}(z-u)\psi^{''}(u)du;x\right).
	\end{align*}
	By Lemma \ref{lemma4} and (\ref{eqn6}),
	\begin{equation*}
		\mathcal{H}_{\eta,\kappa}^{(\alpha,\beta)}(z-x;x)=0.
	\end{equation*}
	\begin{align*}
		\mbox{Therefore},\quad|\mathcal{H}_{\eta,\kappa}^{(\alpha,\beta)}(\psi;x)-\psi(x)|&\leq\mathcal{H}_{\eta,\kappa}^{(\alpha,\beta)}\left(\int\limits_{x}^{z}(z-u)|\psi^{''}(u)|du;x\right)\\
		&\leq\|\psi^{''}\|\left|\mathcal{H}_{\eta,\kappa}^{(\alpha,\beta)}\left(\int\limits_{x}^{z}(z-u)du;x\right)\right|.
		%&=\|\psi^{''}\|\left|\mathcal{H}_{\eta,\kappa}^{(\alpha,\beta)}\left(\int\limits_{x}^{z}(z-u)du;x\right)\right|
	\end{align*}
	By (\ref{eqn6}),
	\begin{align*}
		|\mathcal{H}_{\eta,\kappa}^{(\alpha,\beta)}(\psi;x)-\psi(x)|&\leq\|\psi^{''}\|\left(\left|(\mathcal{R}_{\eta,\kappa}^{(\alpha,\beta)}\left(\int\limits_{x}^{z}(z-u)du;x\right)\right|+\left|\int_{x}^{x+\frac{1+\alpha}{\eta}}\left(x+\frac{1+\alpha}{\eta}-u\right)du\right|\right)
	\end{align*}
	\begin{equation}\label{eqn8}
		\leq\|\psi^{''}\|\left(\mu_{\eta,2}(x)+\left(\frac{1+\alpha}{\eta}\right)^{2}\right)=\delta\|\psi^{''}\|.
	\end{equation}
	\begin{align*}
		\mbox{We have, }\quad\mathcal{R}_{\eta,\kappa}^{(\alpha,\beta)}(\Phi;x)-\Phi(x)&=
		\mathcal{H}_{\eta,\kappa}^{(\alpha,\beta)}(\Phi-\psi;x)-(\Phi-\psi)(x)+\mathcal{H}_{\eta,\kappa}^{(\alpha,\beta)}(\psi;x)-\psi(x)\\&\quad+\Phi\left(x+\frac{1+\alpha}{\eta}\right)-\Phi(x).
	\end{align*}
	%		Form (\ref{eqn2}), (\ref{eqn5}), (\ref{eqn8}) and Lemma \ref{lemma3}, we get
	\begin{align*}
		|\mathcal{R}_{\eta,\kappa}^{(\alpha,\beta)}(\Phi;x)-\Phi(x)|&\leq|\mathcal{H}_{\eta,\kappa}^{(\alpha,\beta)}(\Phi-\psi;x)-(\Phi-\psi)(x)|+|\mathcal{H}_{\eta,\kappa}^{(\alpha,\beta)}(\psi;x)-\psi(x)|\\&\quad+\left|\Phi\left(x+\frac{1+\alpha}{\eta}\right)-\Phi(x)\right|\\
		&\leq 2\|\Phi-\psi\|+\delta \|\psi^{''}\|+\omega\left(\Phi,\frac{1+\alpha}{\eta}\right).
	\end{align*}
	Taking infimum of $\psi$ on $W_{C^{2}([0,\infty))}$ of the right hand side of the inequality,
	\begin{align*}
		|\mathcal{R}_{\eta,\kappa}^{(\alpha,\beta)}(\Phi;x)-\Phi(x)|&\leq K_{2}\left(\Phi,\frac{\delta}{2}\right)+\omega\left(\Phi,\frac{1+\alpha}{\eta}\right).
	\end{align*}
	By (\ref{eqn4}),
	\begin{align*}
		|\mathcal{R}_{\eta,\kappa}^{(\alpha,\beta)}(\Phi;x)-\Phi(x)|&\leq \mathfrak{M}\omega_{2}(\Phi,\sqrt{\delta/2})+\omega\left(\Phi,\frac{1+\alpha}{\eta}\right).
	\end{align*}
\end{proof}
Here, we study ways to apply the Ditzian-toitik moduli of smoothness to find the degree of approximation:
\begin{align*}
	\mbox{By \cite{z.ditzian}, let}\quad\omega_{\phi^{\lambda}}^{2}(\Phi,\delta)&=\sup_{0<\epsilon\leq\delta}\left(\sup_{x,x+\epsilon\phi^{\lambda},x-2\epsilon\phi^{\lambda}\in [0,\infty)}\left|\Phi(x+\epsilon\phi^{\lambda})-2\Phi(x)+\Phi(x-\epsilon\phi^{\lambda})\right|\right),
\end{align*}
and the corresponding K-functional is given by
\begin{equation*}
	K_{2,\phi^{\lambda}}(\Phi,\delta^{2})=\inf_{\psi^{'}\in A.C._{loc}([0,\infty))}\{\|\Phi-\psi\|+\delta^{2}\|\phi^{2\lambda}\psi^{''}\|\},
\end{equation*}
\begin{equation*}
	\mbox{and}\;\;\mathfrak{D}_{\lambda}^{2}=\{\psi\in \mathcal{C}_{[0,\infty)}^{\mathcal{B}}:\psi^{'}\in A.C._{loc}([0,\infty)), \|\phi^{2\lambda}\psi^{''}\|<\infty \},\;\mbox{where}\;\phi^{2}(x)=x, 0\leq\lambda\leq1.
\end{equation*}
We have,
  $$K_{2,\phi^{\lambda}}(\Phi,\delta^{2})\sim \omega_{\phi^{\lambda}}^{2}(\Phi,\delta).$$
\begin{theorem}
	For $\Phi\in \mathcal{C}_{[0,\infty)}^{\mathcal{B}}$ and $x\in[0,\infty)$, then
	\begin{equation*}
		|\mathcal{R}_{\eta,\kappa}^{(\alpha,\beta)}(\Phi;x)-\Phi(x)|
		\leq4\omega_{\phi^{\lambda}}^{2}\left(\Phi,\frac{\delta_{\eta}^{(1-\lambda)}(x)}{\sqrt{2n}}\right)+\omega\left(\Phi,\frac{1+\alpha}{\eta}\right).
	\end{equation*}
\end{theorem}
\begin{proof}
	Consider
	\begin{equation}\label{eqn9}
		\mathcal{H}_{\eta,\kappa}^{(\alpha,\beta)}(\Phi;x)=\mathcal{R}_{\eta,\kappa}^{(\alpha,\beta)}(\Phi;x)-\Phi\left(x+\frac{1+\alpha}{\eta}\right)+\Phi(x).
	\end{equation}
	Let $\psi\in \mathfrak{D}_{\lambda}^{2}$, then by Taylor's theorem
	\begin{equation}\label{eqn10}
		\psi(z)=\psi(x)+(z-x)\psi^{'}(x)+\frac{1}{2}\int\limits_{x}^{z}(z-u)\psi^{''}(u)du.
	\end{equation}
	Applying $\mathcal{H}_{\eta,\kappa}^{(\alpha,\beta)}$ on (\ref{eqn10}),
	\begin{equation}\label{eqn11}
		\mathcal{H}_{\eta,\kappa}^{(\alpha,\beta)}(\psi;x)=\psi(x)+\psi^{'}(x)\mathcal{H}_{\eta,\kappa}^{(\alpha,\beta)}(z-x;x)+\frac{1}{2}\mathcal{H}_{\eta,\kappa}^{(\alpha,\beta)}\left(\int\limits_{x}^{z}(z-u)\psi^{''}(u)du;x\right).
	\end{equation}
	By Lemma \ref{lemma4} and (\ref{eqn9}),
	\begin{equation*}
		\mathcal{H}_{\eta,\kappa}^{(\alpha,\beta)}(z-x;x)=0.
	\end{equation*}
	We have,
	\begin{equation}\label{eqn12}
		|\mathcal{H}_{\eta,\kappa}^{(\alpha,\beta)}(\Phi;x)|\leq3\|\Phi\|,
	\end{equation}
	and
	\begin{align*}
		\mu_{\eta,2}(x)&
		\leq\frac{1}{\eta}\delta_{\eta}^{2}(x),
	\end{align*}
	where $\delta_{\eta}^{2}(x)=\frac{\phi^{2}(x)(3\beta+1)+\alpha^{2}\beta+4\alpha\beta+3\beta+\alpha+1}{\beta}$.\\
	From (\cite{z.ditzian}, p. 141), for $z<u<x$, we have
	\begin{equation}\label{eqn13}
		\frac{|z-u|}{\phi^{2\lambda}(u)}\leq\frac{|z-x|}{\phi^{2\lambda}(x)}\;\mbox{and}\; \frac{|z-u|}{\delta_{\eta}^{2\lambda}(u)}\leq\frac{|z-x|}{\delta_{\eta}^{2\lambda}(x)}.
	\end{equation}
	By (\ref{eqn9}) and (\ref{eqn11}),
	\begin{align*}
		|\mathcal{H}_{\eta,\kappa}^{(\alpha,\beta)}(\psi;x)-\psi(x)|&
		\leq\left|\mathcal{H}_{\eta,\kappa}^{(\alpha,\beta)}\left(\int\limits_{x}^{z}(z-u)\psi^{''}(u)du;x\right)\right|\\
		&\leq\left|\mathcal{R}_{\eta,\kappa}^{(\alpha,\beta)}\left(\int\limits_{x}^{z}(z-u)\psi^{''}(u)du;x\right)\right|+\left|\int_{x}^{x+\frac{1+\alpha}{\eta}}\left(x+\frac{1+\alpha}{\eta}-u\right)\psi^{''}(u)du\right|,
	\end{align*}
	using (\ref{eqn13}),
	\begin{align*}
		\leq\|\delta_{\eta}^{2\lambda}\psi^{''}\|\mathcal{R}_{\eta,\kappa}^{(\alpha,\beta)}\left(\frac{(z-x)^{2}}{\delta_{\eta}^{2\lambda}(x)};x\right)+\frac{\|\delta_{\eta}^{2\lambda}\psi^{''}\|}{\delta_{\eta}^{2\lambda}(x)}\left(\frac{1+\alpha}{\eta}\right)^{2}.
	\end{align*}
	\begin{align*}
		|\mathcal{H}_{\eta,\kappa}^{(\alpha,\beta)}(\psi;x)-\psi(x)|&\leq\delta_{\eta}^{-2\lambda}(x)\|\delta_{\eta}^{2\lambda}\psi^{''}\|\left(\mu_{\eta,2}(x)\right)+\delta_{\eta}^{-2\lambda}(x)\|\delta_{\eta}^{2\lambda}\psi^{''}\|\left(\mathcal{R}_{\eta,\kappa}^{(\alpha,\beta)}((z-x);x)\right)^{2}\\
		&\leq\delta_{\eta}^{-2\lambda}(x)\|\delta_{\eta}^{2\lambda}\psi^{''}\|\frac{\delta_{\eta}^{2}(x)}{\eta}+\delta_{\eta}^{-2\lambda}(x)\|\delta_{\eta}^{2\lambda}\psi^{''}\|\frac{\delta_{\eta}^{2}(x)}{\eta}.
	\end{align*}
	Hence,
	\begin{equation}\label{eqn14}
		|\mathcal{H}_{\eta,\kappa}^{(\alpha,\beta)}(\psi;x)-\psi(x)|\leq\frac{2\delta_{\eta}^{2(1-\lambda)}(x)}{\eta}\|\delta_{\eta}^{2\lambda}\psi^{''}\|.
	\end{equation}
	Using (\ref{eqn2}), (\ref{eqn12}) and  (\ref{eqn14}),
	\begin{align*}
		|\mathcal{H}_{\eta,\kappa}^{(\alpha,\beta)}(\Phi;x)-\Phi(x)|&\leq|\mathcal{H}_{\eta,\kappa}^{(\alpha,\beta)}(\Phi-\psi;x)|+|\mathcal{H}_{\eta,\kappa}^{(\alpha,\beta)}(\psi;x)-\psi(x)|+|\Phi(x)-\psi(x)|\\
		&\leq4\|\Phi-\psi\|+|\mathcal{H}_{\eta,\kappa}^{(\alpha,\beta)}(\psi;x)-\psi(x)|\\
		&\leq4\|\Phi-\psi\|+\frac{2\delta_{\eta}^{2(1-\lambda)}(x)}{\eta}\|\delta_{\eta}^{2\lambda}\psi^{''}\|.
	\end{align*}
	Hence,
	\begin{align*}
		|\mathcal{R}_{\eta,\kappa}^{(\alpha,\beta)}(\Phi;x)-\Phi(x)|&\leq |\mathcal{H}_{\eta,\kappa}^{(\alpha,\beta)}(\Phi;x)-\Phi(x)|+\left|\Phi\left(x+\frac{1+\alpha}{\eta}\right)-\Phi(x)\right|\\
		&\leq4\omega_{\phi^{\lambda}}^{2}\left(\Phi,\frac{\delta_{\eta}^{(1-\lambda)}(x)}{\sqrt{2\eta}}\right)+\omega\left(\Phi,\frac{1+\alpha}{\eta}\right)
	\end{align*}
\end{proof}

Now, we discuss Voronovskaja asymptotic formula (\cite{kapil}) for the operators $\mathcal{R}_{\eta,\kappa}^{(\alpha,\beta)}(\Phi;x)$.
\begin{theorem}
	For functions $\Phi,\Phi^{'},\Phi^{''}\in \mathcal{C}_{[0,\infty)}^{\mathcal{B}}$ and $x\in[0,\infty)$, then
	\begin{equation*}
		\lim_{\eta\rightarrow\infty}\eta[\mathcal{R}_{\eta,\kappa}^{(\alpha,\beta)}(\Phi;x)-\Phi(x)]=(1+\alpha)\Phi^{'}(x)+\frac{x(3\beta+1)}{2\beta}\Phi^{''}(x).
	\end{equation*}
\end{theorem}
\begin{proof}
	By the Taylor's theorem
	\begin{equation}\label{eqn15}
		\Phi(z)=\Phi(x)+(z-x)\Phi^{'}(x)+\frac{(z-x)^{2}}{2}\Phi^{''}(x)+\epsilon_{\mathfrak{B}}(x;z)(z-x)^{2},
	\end{equation}\\where $\epsilon_{\mathfrak{B}}(x;z)$ is the Peano form of the remainder,\;$\epsilon_{\mathfrak{B}}(x;z)\in \mathcal{C}_{[0,\infty)}^{\mathcal{B}}.$\\
	We have,\hspace{2cm}$\epsilon_{\mathfrak{B}}(x;z)=\frac{\Phi(z)-\Phi(x)-(z-x)\Phi^{'}(x)-\frac{(z-x)^{2}}{2!}\Phi^{''}(x)}{(z-x)^{2}}$,\\and
	\begin{equation}\label{eqn16}
		\lim_{z\rightarrow x}\epsilon_{\mathfrak{B}}(x;z)=0.
	\end{equation}
	Applying $\mathcal{R}_{\eta,\kappa}^{(\alpha,\beta)}$ on (\ref{eqn15}),
	\begin{align*}
		\mathcal{R}_{\eta,\kappa}^{(\alpha,\beta)}(\Phi(z);x)
		&=\Phi(x)+\Phi^{'}(x)\mu_{\eta,1}(x)+\frac{\Phi^{''}(x)}{2}\mu_{\eta,2}(x)\\&\quad+\mathcal{R}_{\eta,\kappa}^{(\alpha,\beta)}(\epsilon_{\mathfrak{B}}(x;z)(z-x)^{2};x).
	\end{align*}
	\begin{align*}
		\lim_{\eta\rightarrow\infty}\eta[\mathcal{R}_{\eta,\kappa}^{(\alpha,\beta)}(\Phi(z);x)-\Phi(x)]&=\Phi^{'}(x)\lim_{\eta\rightarrow\infty}(\eta\mu_{\eta,1}(x))\\&\quad+\frac{\Phi^{''}(x)}{2}\lim_{\eta\rightarrow\infty}(\eta\mu_{\eta,2}(x))\\&\quad+\lim_{\eta\rightarrow\infty}(\eta\mathcal{R}_{\eta,\kappa}^{(\alpha,\beta)}(\epsilon_{\mathfrak{B}}(x;z)(z-x)^{2};x))\\
		&=(1+\alpha)\Phi^{'}(x)+\frac{x(3\beta+1)}{2\beta}\Phi^{''}(x)+\mathcal{R}_{\eta},
	\end{align*}
	where $\mathcal{R}_{\eta}=\displaystyle\lim_{\eta\rightarrow\infty}(\eta\mathcal{R}_{\eta,\kappa}^{(\alpha,\beta)}(\epsilon_{\mathfrak{B}}(x;z)(z-x)^{2};x))$.\\Using Cauchy-Bunyakovsky-Schwarz Inequality,
	\begin{equation*}
		\eta\mathcal{R}_{\eta,\kappa}^{(\alpha,\beta)}(\epsilon_{\mathfrak{B}}(x;z)(z-x)^{2};x)\leq \sqrt{\eta^{2}\mathcal{R}_{\eta,\kappa}^{(\alpha,\beta)}(\epsilon_{\mathfrak{B}}^{2}(x;z);x)}\sqrt{\mu_{\eta,4}(x)}.
	\end{equation*}
	We observe that by (\ref{eqn16}) if $\eta\rightarrow\infty$, then $z\rightarrow x$ and $\displaystyle\lim_{z\rightarrow x}\epsilon_{\mathfrak{B}}(x;z)=0$. It follows that
	\begin{equation*}
		\lim_{\eta\rightarrow\infty}(\eta^{2}\mathcal{R}_{\eta,\kappa}^{(\alpha,\beta)}(\epsilon_{\mathfrak{B}}^{2}(x;z);x))=0\;\mbox{uniformly with respect to}\;x\in[0,\infty).
	\end{equation*}
	So,\;$R_{\eta}=\displaystyle\lim_{\eta\rightarrow\infty}(\eta\mathcal{R}_{\eta,\kappa}^{(\alpha,\beta)}(\epsilon_{\mathfrak{B}}(x;z)(z-x)^{2};x))=0$.\\
	Therefore, the desired result is proved.
\end{proof}
\begin{theorem}
	For $\Phi\in\mathcal{C}_{[0,\infty)}^{\mathcal{B}}$ and $x\in[0,\infty)$, then
	\begin{equation*}
		|\mathcal{R}_{\eta,\kappa}^{(\alpha,\beta)}(\Phi;x)-\Phi(x)|\leq2\omega(\Phi,\delta),\;\mbox{where}\;\delta=\sqrt{\mu_{\eta,2}(x)}.
	\end{equation*}
\end{theorem}
\begin{proof}
	By the property of modulus of continuity,
	\begin{align*}
		|\Phi(z)-\Phi(x)|&\leq\omega(\Phi,|z-x|)\\
		&\leq\left(1+\frac{1}{\delta}|z-x|\right)\omega(\Phi,\delta).
	\end{align*}
	\begin{align*}
		|\mathcal{R}_{\eta,\kappa}^{(\alpha,\beta)}(\Phi;x)-\Phi(x)|&\leq p_{\eta,\kappa}(x)\int_{0}^{\infty}\mathcal{I}_{\kappa,\eta}^{\beta}(z)\left|\Phi(z)-\Phi(x)\right|dz\\
		&\leq\left(1+\frac{1}{\delta}p_{\eta,\kappa}(x)\int_{0}^{\infty}\mathcal{I}_{\kappa,\eta}^{\beta}(z)\left|z-x\right|dz\right)\omega(\Phi,\delta),
	\end{align*}
 where $p_{\eta,\kappa}(x)=\exp{\left(\frac{-\eta x}{2}\right)}2^{-\alpha-1}\sum_{\kappa=0}^{\infty}2^{-\kappa}\mathcal{L}_{\kappa}^{(\alpha)}\left(\frac{-\eta x}{2}\right)$.\\
	By Cauchy-Bunyakovsky-Schwarz inequality,
	\begin{align*}
		|\mathcal{R}_{\eta,\kappa}^{(\alpha,\beta)}(\Phi;x)-\Phi(x)|&\leq1+\left(\frac{1}{\delta}\sqrt{p_{\eta,\kappa}(x)\int_{0}^{\infty}\mathcal{I}_{\kappa,\eta}^{\beta}(z)dz}\sqrt{p_{\eta,\kappa}(x)\int_{0}^{\infty}\mathcal{I}_{\kappa,\eta}^{\beta}(z)(z-x)^{2}dz}\right)\\&\quad\times\omega(\Phi,\delta)\\
		&\quad=\left(1+\frac{1}{\delta}\sqrt{(\mathcal{R}_{\eta,\kappa}^{(\alpha,\beta)}(\Phi;x)((z-x)^{2};x)}\right)\omega(\Phi,\delta)\\
		&\quad=\left(1+\frac{1}{\delta}\sqrt{\mu_{\eta,2}(x)}\right)\omega(\Phi,\delta)=2\omega(\Phi,\delta).
	\end{align*}
\end{proof}
Now, the rate of convergence of $\mathcal{R}_{\eta,\kappa}^{(\alpha,\beta)}(\Phi;x)$ by the Lipschitz class $Lip_{\mathfrak{K}}(\tau),\;\tau>0$ is obtained. 
If $\Phi\in Lip_{\mathfrak{K}}(\tau)$, then function $\Phi$ satisfies the inequality 
\begin{equation}\label{eqn17}|\Phi(z)-\Phi(x)|\leq \mathfrak{K}|z-x|^{\tau},\quad x,z\in[0,\infty),
\end{equation}
where $\mathfrak{K}$ is a positive constant.
\begin{theorem}If $x\in[0,\infty)$ and $\Phi\in \mathcal{C}_{[0,\infty)}^{\mathcal{B}}$ belongs to the class $Lip_{\mathfrak{K}}(\tau)$, then
	\begin{equation*}
		|\mathcal{R}_{\eta,\kappa}^{(\alpha,\beta)}(\Phi;x)-\Phi(x)|\leq \mathfrak{K}(\mu_{\eta,2}(x))^{\frac{\eta}{2}},
	\end{equation*}
	where $\mu_{\eta,2}(x)=\frac{x(3\beta+1)}{\beta\eta}+\frac{\alpha^{2}\beta+4\alpha\beta+3\beta+\alpha+1}{\beta\eta^{2}}$.
\end{theorem}
\begin{proof}
	By (\ref{eqn17}),
	\begin{align*}|\mathcal{R}_{\eta,\kappa}^{(\alpha,\beta)}(\Phi;x)-\Phi(x)|&\leq \mathcal{R}_{\eta,\kappa}^{(\alpha,\beta)}(|\Phi(z)-\Phi(x)|;x) \\&\leq \mathcal{R}_{\eta,\kappa}^{(\alpha,\beta)}(\mathfrak{K}|z-x|^{\tau};x).\end{align*}
	By Cauchy-Bunyakovsky-Schwarz Inequality,
	\begin{align*}
		|\mathcal{R}_{\eta,\kappa}^{(\alpha,\beta)}(\Phi;x)-\Phi(x)|&\leq\mathfrak{K}[\mathcal{R}_{\eta,\kappa}^{(\alpha,\beta)}((z-x)^{2};x)]^{\frac{\tau}{2}}.
	\end{align*}
	using Lemma \ref{lemma4},
	\begin{align*}|\mathcal{R}_{\eta,\kappa}^{(\alpha,\beta)}(\Phi;x)-\Phi(x)|&\leq \mathfrak{K}\left[\frac{x(3\beta+1)}{\beta\eta}+\frac{\alpha^{2}\beta+4\alpha\beta+3\beta+\alpha+1}{\beta\eta^{2}}\right]^{\frac{\tau}{2}}= \mathfrak{K}(\mu_{\eta,2}(x))^{\frac{\tau}{2}}.
	\end{align*}
\end{proof}
\section{Weighted approximation}
The goal of this section is dedicated to finding approximation properties of the operators $\mathcal{R}_{\eta,\kappa}^{(\alpha,\beta)}$ corresponding to weighted spaces of continuous functions.\\ 
Let $\mathcal{B}_{\zeta}([0,\infty))$ be a normed space defined by
\begin{equation*}
	\mathcal{B}_{\zeta}([0,\infty))=\{\Phi:[0,\infty)\rightarrow \mathbb{R}:|\Phi(x)|\leq \mathcal{M}_{\Phi}\zeta(x),\;x\in[0,\infty)\},
\end{equation*}
endowed with the norm
\begin{equation*}
	\|\Phi\|_{\zeta}=\sup_{x\in[0,\infty)}\frac{|\Phi(x)|}{\zeta(x)},
\end{equation*}
where $\zeta(x)=1+x^{2}$ and constant $\mathcal{M}_{\Phi}>0$ depends on function $\Phi$.\\
Additionally, we define following spaces,
\begin{enumerate} \item[i.)]$\mathcal{C}_{\zeta}([0,\infty))=\{\Phi\in\mathcal{B}_{\zeta}([0,\infty)):\Phi\;\mbox{is continuous function on}\;[0,\infty)\}$\\
	\item[ii.)]$\displaystyle\mathcal{C}_{\zeta}^{*}([0,\infty))=\{\Phi\in\mathcal{C}_{\zeta}([0,\infty)):\lim_{x\rightarrow\infty}\frac{\Phi(x)}{\zeta(x)}\;\mbox{exists in}\;\mathbb{R}\}$
\end{enumerate}
It was shown in \cite{ispir2001modified} that for any $\Phi\in\displaystyle\mathcal{C}_{\zeta}^{*}([0,\infty))$, the weighted modulus of continuity is denoted by 
\begin{equation}\label{eqn18}
	\Omega(\Phi;\delta)=\sup_{0\leq\epsilon<\delta, x\in[0,\infty)}\frac{|\Phi(x+\epsilon)-\Phi(x)|}{(1+\epsilon^{2})\zeta(x)}.
\end{equation}
\begin{lemma}\label{lemma6}
	If $\Phi\in\displaystyle\mathcal{C}_{\zeta}([0,\infty))$, then
	\begin{enumerate}
		\item [(i)] $\Omega(\Phi;\delta)$ is a monotonically increasing function of $\delta$,
		\item [(ii)] $\Omega(\Phi;\lambda\delta)\leq2(1+\lambda)(1+\delta^{2})\Omega(\Phi;\delta), \lambda>0$,
  	\item [(iii)]$\Omega(\Phi;\delta)\rightarrow0\;\mbox{as}\;\delta\rightarrow0$.
		
	\end{enumerate}
	
		The weighted modulus of continuity definition enables us to write 
	\begin{equation}\label{eqn19}
		|\Phi(z)-\Phi(x)|\leq\zeta(x)(1+(z-x)^{2})\Omega(\Phi;|z-x|).
	\end{equation}
\end{lemma}
\begin{theorem}
	For $\Phi\in \mathcal{C}_{[0,\infty)}^{\mathcal{B}}$ and $x\in[0,\infty)$, then
	\begin{align*}
		|\mathcal{R}_{\eta,\kappa}^{(\alpha,\beta)}(\Phi;x)-\Phi(x)|&\leq2\left(1+\frac{1}{\eta}\right)\zeta(x)\Omega\left(\Phi;\frac{1}{\sqrt{\eta}}\right)\left(1+\mathcal{E}_{1}\frac{x(3\beta+1)}{\beta\eta}+\sqrt{\mathcal{E}_{1}\frac{x(3\beta+1)}{\beta}}\right.\\
		&\left.\quad\times\left(1+\frac{1}{\eta}\sqrt{\mathcal{E}_{2}\frac{x^{2}(3\beta+1)^{2}}{\beta^{2}}}\right)\right).
	\end{align*}
\end{theorem}
\begin{proof}
	For $x,z\in[0,\infty)$. From (\ref{eqn19}) and Lemma \ref{lemma6}, we get \\
	\begin{align}\label{eqn20}
		|\Phi(z)-\Phi(x)|&\leq (1+(z-x)^{2})\zeta(x)\Omega\left(\Phi;\frac{|z-x|\delta}{\delta}\right)\nonumber\\
		&\leq2(1+\delta^{2})\zeta(x)\left(1+\frac{|z-x|}{\delta}\right)(1+(z-x)^{2})\Omega(\Phi;\delta).
	\end{align}
	Applying $\mathcal{R}_{\eta,\kappa}^{(\alpha,\beta)}$ on (\ref{eqn20}),
	\begin{align*}
		|\mathcal{R}_{\eta,\kappa}^{(\alpha,\beta)}(\Phi;x)-\Phi(x)|&\leq2(1+\delta^{2})\zeta(x)\Omega(\Phi;\delta)\mathcal{R}_{\eta,\kappa}^{(\alpha,\beta)}\left[\left(1+\frac{|z-x|}{\delta}\right)(1+(z-x)^{2});x\right]\\
		&=2(1+\delta^{2})\zeta(x)\Omega(\Phi;\delta)\left[1+\mu_{\eta,2}(x)+\frac{1}{\delta}\mathcal{R}_{\eta,\kappa}^{(\alpha,\beta)}(|z-x|;x)\right.\\
		&\left.\quad+\frac{1}{\delta}\mathcal{R}_{\eta,\kappa}^{(\alpha,\beta)}(|z-x|(z-x)^{2};x)\right].
	\end{align*}
	
	By Cauchy-Bunyakovsky-Schwarz Inequality,
	\begin{align}\label{eqn21}
		|\mathcal{R}_{\eta,\kappa}^{(\alpha,\beta)}(\Phi;x)-\Phi(x)|&\leq2(1+\delta^{2})\zeta(x)\Omega(\Phi;\delta)\left[1+\mu_{\eta,2}(x)\right.\nonumber\\
		&\left.+\frac{1}{\delta}(\mu_{\eta,2}(x))^{1/2}+\frac{1}{\delta}\left(\mu_{\eta,2}(x)\right)^{1/2}\left(\mu_{\eta,4}(x)\right)^{1/2}\right].
	\end{align}
	From Lemma (\ref{lemma4}), we have\\
	$$\mu_{\eta,2}(x)\leq\frac{\mathcal{E}_{1}x(3\beta+1)}{\beta\eta},$$
	and
	$$\mu_{\eta,4}(x)\leq\frac{\mathcal{E}_{2}x^{2}(3\beta+1)^{2}}{\beta^{2}\eta^{2}},$$
	
	where $\mathcal{E}_{1}>1$ and $\mathcal{E}_{2}>1$ are constants.\\
	Using the above inequalities in (\ref{eqn21}) and choosing $\delta=\frac{1}{\sqrt{\eta}}$, we get
	\begin{align*}
		|\mathcal{R}_{\eta,\kappa}^{(\alpha,\beta)}(\Phi;x)-\Phi(x)|&\leq2\left(1+\frac{1}{\eta}\right)\zeta(x)\Omega\left(\Phi;\frac{1}{\sqrt{\eta}}\right)\left(1+\mathcal{E}_{1}\frac{x(3\beta+1)}{\beta\eta}+\sqrt{\mathcal{E}_{1}\frac{x(3\beta+1)}{\beta}}\right.\\
		&\left.\quad\times\left(1+\frac{1}{\eta}\sqrt{\mathcal{E}_{2}\frac{x^{2}(3\beta+1)^{2}}{\beta^{2}}}\right)\right).
	\end{align*}
\end{proof}		
\begin{theorem}\label{theorem7}
	For each $\Phi\in\mathcal{C}_{\zeta}^{*}([0,\infty))$, then
	\begin{equation*}
		\lim_{\eta\rightarrow\infty}\|\mathcal{R}_{\eta,\kappa}^{(\alpha,\beta)}(\Phi;.)-\Phi\|_{\zeta}=0.
	\end{equation*}
\end{theorem}
\begin{proof}
	Using \cite{gadjiev1976theorems}, to prove this theorem, it is sufficient to verify the following conditions
	\begin{equation}\label{eqn22}
		\lim_{\eta\rightarrow\infty}\|\mathcal{R}_{\eta,\kappa}^{(\alpha,\beta)}(\varrho_{\jmath};x)-x^{\jmath}\|_{\zeta}=0,\;\jmath=0,1,2.
	\end{equation}
	Since $\mathcal{R}_{\eta,\kappa}^{(\alpha,\beta)}(\varrho_{0};x)=1$, so for $\jmath=0$ (\ref{eqn22}) holds.\\
	By Lemma \ref{lemma3},
	\begin{align*}
		\|\mathcal{R}_{\eta,\kappa}^{(\alpha,\beta)}(\varrho_{1};x)-x\|_{\zeta}&= \sup_{x\in[0,\infty)}\frac{|\mathcal{R}_{\eta,\kappa}^{(\alpha,\beta)}(\varrho_{1};x)-x|}{\zeta(x)}\\
		&=\frac{1}{\eta}\sup_{x\in[0,\infty)}\left(\frac{1+\alpha}{\zeta(x)}\right)\rightarrow 0\;\mbox{as}\;\eta\rightarrow\infty.\\
	\end{align*}
	The condition (\ref{eqn22}) holds for $\jmath=1$.\\
	Again by Lemma \ref{lemma3},
	\begin{align*}
		\|\mathcal{R}_{\eta,\kappa}^{(\alpha,\beta)}(\gamma_{2};x)-x^{2}\|_{\zeta}&= \sup_{x\in[0,\infty)}\frac{|\mathcal{R}_{\eta,\kappa}^{(\alpha,\beta)}(\gamma_{2};x)-x^{2}|}{\zeta(x)}\\
		&=\frac{1}{\eta^{2}}\sup_{x\in[0,\infty)}\left(\frac{x(2\alpha\eta+5\eta)+\alpha^{2}+4\alpha+3}{\zeta(x)}\right).\\
	\end{align*}
	Clearly,\;$\|\mathcal{R}_{\eta,\kappa}^{(\alpha,\beta)}(\varrho_{2};x)-x^{2}\|_{\zeta}\rightarrow 0\;\mbox{as}\;\eta\rightarrow\infty$,
	the condition (\ref{eqn22}) holds for $\jmath=2$.\\
	Hence the theorem proved.
\end{proof}
\begin{theorem}
	For each $\Phi\in\mathcal{C}_{\zeta}([0,\infty))$ and $\ell>1$, then
	\begin{equation*}
		\lim_{\eta\rightarrow\infty}\sup_{x\in[0,\infty)}\frac{|\mathcal{R}_{\eta,\kappa}^{(\alpha,\beta)}(\Phi;x)-\Phi(x)|}{(\zeta(x))^{\ell}}=0.
	\end{equation*}
\end{theorem}
\begin{proof}
	For any fixed $x_{0}>0$,
	\begin{align*}
		\sup_{x\in[0,\infty)}\frac{|\mathcal{R}_{\eta,\kappa}^{(\alpha,\beta)}(\Phi;x)-\Phi(x)|}{(\zeta(x))^{\ell}}&\leq\sup_{x\leq x_{0}}\frac{|\mathcal{R}_{\eta,\kappa}^{(\alpha,\beta)}(\Phi;x)-\Phi(x)|}{(\zeta(x))^{\ell}}+\sup_{x\geq x_{0}}\frac{|\mathcal{R}_{\eta,\kappa}^{(\alpha,\beta)}(\Phi;x)-\Phi(x)|}{(\zeta(x))^{\ell}}\\
		&\leq\|\mathcal{R}_{\eta,\kappa}^{(\alpha,\beta)}(\Phi;.)-\Phi\|_{\zeta^{\ell}}+\|\Phi\|_{\zeta}\sup_{x\geq x_{0}}\frac{|\mathcal{R}_{\eta,\kappa}^{(\alpha,\beta)}(1+z^{2};x)|}{(\zeta(x))^{\ell}}\\&\quad+\sup_{x\geq x_{0}}\frac{|\Phi(x)|}{(\zeta(x))^{\ell}}.
	\end{align*}
	By Theorem \ref{theorem7} the second term from the left in the above inequality tends to $0$ as $\eta\rightarrow\infty$ and for the fixed  $x_{0}$, if we choose $x_{0}$ large enough, then terms $\displaystyle\|\Phi\|_{\zeta}\sup_{x\geq x_{0}}\frac{|\mathcal{R}_{\eta,\kappa}^{(\alpha,\beta)}(1+z^{2};x)|}{(\zeta(x))^{\ell}}$ and $\displaystyle\sup_{x\geq x_{0}}\frac{|\Phi(x)|}{(\zeta(x))^{\ell}}$ can be made small enough. Thus the desired proof is proved.
\end{proof}
\section{Graphical and numerical analysis}
The graphical representations present convergence of the proposed operators on the intervals $[0,1]$ and $[0,2]$ to $\Phi(x)$ for the different values of parameters $\alpha$ and $\beta$ in the Figures \ref{figure1} and \ref{figure3}. The absolute error $\epsilon_{\eta}^{(\alpha,\beta)}(x)=|\mathcal{R}_{\eta,\kappa}^{(\alpha,\beta)}(\Phi;x)-\Phi(x)|$ to $\Phi(x)$ for different values of $x$ over  $[0,2]$ is calculated in the Tables \ref{table1} and \ref{table2}. Errors are represented graphically in Figures \ref{figure2} and \ref{figure4}. 
\begin{example}
	Let us consider test function $\Phi(x)=e^{-5x}x$ and $\alpha=1, \beta=0.98, \eta=25, 50, 75, 100$.
\end{example}
\begin{figure}[ht]
	\centering
	\includegraphics[width=5cm]{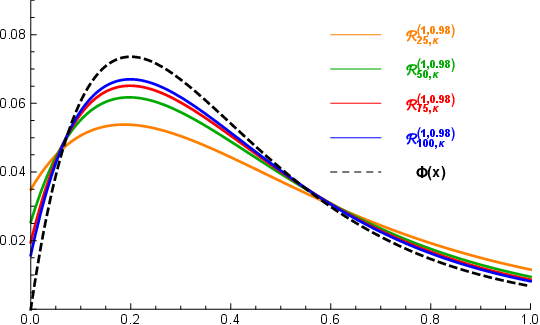}
	\caption{Convergence of $\mathcal{R}_{25,\kappa}^{(1,0.98)}(orange),\mathcal{R}_{50,\kappa}^{(1,0.98)}(green),\mathcal{R}_{75,\kappa}^{(1,0.98)}(red),\mathcal{H}_{100,\kappa}^{(1,0.98)}(blue)$ and test function(black)}
	\label{figure1}
\end{figure}
\begin{figure}[ht]
	\centering
	\includegraphics[width=5cm]{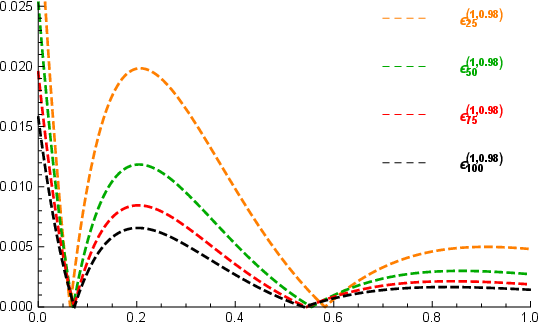}
	\caption{Graph of errors $\epsilon_{25}^{(1,0.98)}(orange),\epsilon_{50}^{(1,0.98)}(green),\epsilon_{75}^{(1,0.98)}(red), \epsilon_{100}^{(1,0.98)}(black)$}
	\label{figure2}
\end{figure}
\begin{table}[ht]
	\caption{Table for absolute error of the operators $\mathcal{R}_{\eta,\kappa}^{(\alpha,\beta)}$}
	\centering
	\begin{tabular}{l|c|c|c|c}\label{table1}
		x & $\eta=25$ & $\eta=50$ & $\eta=75$ & $\eta=100$ \\
		&$\alpha=1, \beta=0.98$&$\alpha=1, \beta=0.98$&$\alpha=1, \beta=0.98$&$\alpha=1, \beta=0.98$\\
		\hline
		0 & 0.034901800 & 0.025423800 & 0.019518100 & 0.015763700 \\
		0.5 & 0.003740950 & 0.001533680 & 0.000898391 & 0.000617602  \\
		1.0 & 0.004815770 & 0.002722240 & 0.001880080 & 0.001433030\\
		1.5 & 0.002000540 & 0.000915633 & 0.000583079 & 0.000425828\\
		2.0 & 0.000534200 & 0.000197393 & 0.000115992 & 0.000081252
		\\\hline
		x & $\eta=25$ & $\eta=50$ & $\eta=75$ & $\eta=100$ \\
		&$\alpha=5, \beta=10$&$\alpha=5, \beta=10$&$\alpha=5, \beta=10$&$\alpha=5, \beta=10$\\\hline
		0 & 0.060720100 & 0.057087900 & 0.048077200 & 0.040729100 \\
		0.5 & 0.016013200 & 0.009493930 & 0.006739870 & 0.005223240  \\
		1.0 & 0.000626519 & 0.000237580 & 0.000133640 & 0.000089514\\
		1.5 & 0.000436015 & 0.000248134 & 0.000171949 & 0.000131331\\
		2.0 & 0.000150698 & 0.000071107 & 0.000045981 & 0.000033881
	\end{tabular}
\end{table}
\begin{example}
	Let us consider test function $\Phi(x)=x^{3}-2x^{2}+3$ and $\alpha=0.5, \beta=2, \eta=25, 50, 75, 100$.
\end{example}
\begin{figure}[ht]
	\centering
	\includegraphics[width=5cm]{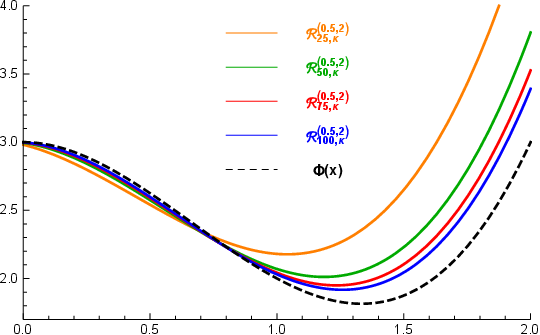}
	\caption{Convergence of $\mathcal{R}_{25,\kappa}^{(0.5,2)}(orange),\mathcal{R}_{50,\kappa}^{(0.5,2)}(green),\mathcal{R}_{75,\kappa}^{(0.5,2)}(red),\mathcal{H}_{100,\kappa}^{(0.5,2)}(blue)$ and test function(blue)}
	\label{figure3}
\end{figure}
\begin{figure}[ht]
	\centering
	\includegraphics[width=5cm]{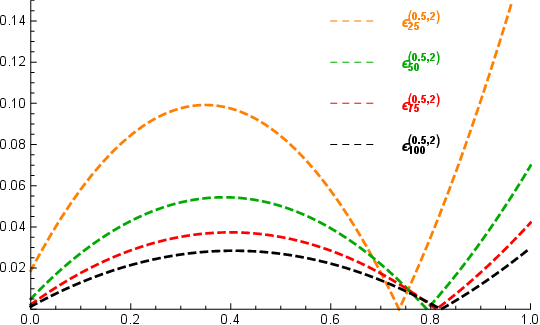}
	\caption{Graph of errors $\epsilon_{25}^{(0.5,2)}(orange),\epsilon_{50}^{(0.5,2)}(green),\epsilon_{75}^{(0.5,2)}(red), \epsilon_{100}^{(0.5,2)}(black)$}
	\label{figure4}
\end{figure}
\begin{table}[ht]
	\caption{Table for absolute error of the operators $\mathcal{R}_{\eta,\kappa}^{(\alpha,\beta)}$}
	\centering
	\begin{tabular}{l|c|c|c|c}\label{table2}
		x & $\eta=25$ & $\eta=50$ & $\eta=75$ & $\eta=100$ \\
		&$\alpha=2, \beta=0.5$&$\alpha=2, \beta=0.5$&$\alpha=2, \beta=0.5$&$\alpha=2, \beta=0.5$\\
		\hline
		0 & 0.06316800 & 0.01869600 & 0.00873956 & 0.00503700\\
		0.5 & 0.10636800 & 0.08449600 & 0.06242840 & 0.04898700 \\
		1.0 & 0.45043200 & 0.14970400 & 0.08388270 & 0.05706300\\
		1.5 & 1.60723000 & 0.68390400 & 0.43019400 & 0.31311300 \\
		2.0 & 3.36403000 & 1.51810000 & 0.97650500 & 0.71916300\\
		\hline
		x & $\eta=25$ & $\eta=50$ & $\eta=75$ & $\eta=100$ \\
		&$\alpha=0.5, \beta=2$&$\alpha=0.5, \beta=2$&$\alpha=0.5, \beta=2$&$\alpha=0.5, \beta=2$\\\hline
		0 & 0.01876800 & 0.00504600 & 0.00229511 & 0.00130575\\
		0.5 & 0.08416800 & 0.05014600 & 0.03511730 & 0.02695580\\
		1.0 & 0.18043200 & 0.06975400 & 0.04206040 & 0.02989420\\
		1.5 & 0.77503200 & 0.35465400 & 0.22923800 & 0.16924400\\
		2.0 & 1.69963000 & 0.80455400 & 0.52641600 & 0.39109400
	\end{tabular}
\end{table}
\newpage
\section{Conclusion}
This study explores the innovative aspects of the integral type of construction of operators adopting modified Laguerre polynomials and P\u{a}lt\u{a}nea basis. It covers various aspects of the operators, including approximation properties, convergence rate, Voronovskaja-type asymptotic formula, and weighted approximation. 
The adaptability and convergence of the proposed operators are vital aspects of the study and depend on the choice of parameters $\alpha, \beta$ and  $\eta$, respectively. It explores how different values of these parameters impact the performance of operators. The graphs are also used to dream up the performance of operators under various selections of $\eta$ and parameters. Graphs provide a more instinctive understanding of how these parameters affect the behavior of proposed operators, which can be especially helpful for conveying findings to others.  
\section*{Declarations}
\textbf{Conflict of interest} The authors declare that there are no conflicts of interest.
	\bibliographystyle{plain}
	\bibliography{nonlinear}
	
\end{document}